\documentclass[11pt,a4paper]{article}

\usepackage{amssymb}
\usepackage{epsfig}
\usepackage{setspace}
\usepackage{amsmath,amssymb,amsthm}
\usepackage{mathrsfs}
\usepackage{color}
\usepackage{colortbl}
\usepackage{array}
\usepackage{braket}

\usepackage{enumerate}
\usepackage{float}

\def\RR{{\bf R}}
\def\ZZ{{\bf Z}}
\def\QQ{{\bf Q}}
\def\Rt{{\bf R}\{t\}}
\def\Ct{{\bf C}\{t\}}

\def\mDelta{{\mit\Delta}}

\numberwithin{equation}{section}

\newcommand{\val}{\mathop{\rm val} }

\newcommand{\diag}{\mathop{\rm diag} }
\newcommand{\supp}{\mathop{\rm supp} }

\newcommand{\Real}{\mathop{\rm Re} }

\newcommand{\sign}{\mathop{\rm sign}}
\newcommand{\Pf}{\mathop{\rm Pf}}

\theoremstyle{plain}
\newtheorem{thm}{Theorem}[section]
\newtheorem{lem}[thm]{Lemma}

\newtheorem{cor}[thm]{Corollary}

\theoremstyle{definition}
\newtheorem{rmk}[thm]{Remark}

\title{A Combinatorial Formula for Principal Minors of a Matrix with Tree-metric Exponents and Its Applications}
\author{Hiroshi Hirai \thanks{
Graduate School of Information Science and Technology,
The University of Tokyo, Tokyo 113-8656, Japan. 
Email:\{hirai, akihiro\_yabe\}@mist.i.u-tokyo.ac.jp
} \and Akihiro Yabe ${}^*$
}

\begin{document}
\maketitle
\begin{abstract}
Let $T$ be a tree with a vertex set $\{ 1,2,\dots, N \}$. Denote by $d_{ij}$ 
the distance between vertices $i$ and $j$.
In this paper, we present an explicit combinatorial formula of principal minors of the matrix $(t^{d_{ij}})$, 
and its applications to tropical geometry, study of multivariate stable polynomials, and representation of valuated matroids. 
We also give an analogous formula for 
a skew-symmetric matrix associated with $T$.
\end{abstract}
\section{Introduction}\label{sec_intro}
Let $T = (V,E)$ be a tree, where $V=\{1,2,\dots, N\}$. 
For $i,j \in V$, denote by $d_{ij}$ 
the number of edges of the unique path between $i$ and $j$ in $T$. 
With an indeterminate $t$, define the matrix $A = (a_{ij})$ by
\begin{align*}
a_{ij} := t^{d_{ij}} \quad (i,j \in V).
\end{align*}
This matrix appeared in the study of the $q$-distance matrix of a tree~\cite{Bapat2006Aqo}. Yan and Yeh \cite{Yan2007Tdo} showed that $\det A$ is given by the following simple formula:
\begin{thm}[Yan--Yeh \cite{Yan2007Tdo}]
$\det A = (1 - t^2)^{N-1}$.
\end{thm}
Our main result can be understood as an extension of Yan--Yeh's formula to principal minors of $A$. 
The motivation of our investigation, however, comes from study of multivariate stable polynomials~\cite{Branden2007Pwt, Branden2010DCa, Choe2004Hmp}, tropical geometry~\cite{Develin2004Tc, Speyer2004TtG}, and representation of valuated matroids~\cite{DressPro1998, Dress1992Vm}. To state our result, let us introduce some notions. 
For $X \subseteq V$, denote by $A[X]$ 
the principal submatrix of $A$ consisting of $a_{ij}$ for $i,j \in X$. 
We say that a forest $F=(V_F, E_F)$ {\it is spanned by} $X$ if $X \subseteq V_F$ and all leaves of $F$ are contained in $X$. Note that the subtree of $T$ spanned by $X$ is the unique minimal subtree including $X$, which is denoted by $T_X = (V_X,E_X)$. 
Define $c(F)$ as the number of connected components of $F$. Denote by ${\rm deg}_{F}(v)$ the degree of a vertex $v$ in $F$. 
Then our main result is the following:
%
\begin{thm}\label{thm_main}
\begin{align}
\det A[X] = \sum_{F}  (-1)^{|X|+ c(F)} t^{2|E_F|}  \prod_{v \in V_{F} \setminus X} ({\rm deg}_F (v) - 1) , \label{eq_main_det_f}
\end{align}
where the sum is taken over all subgraphs $F$ of $T$ spanned by $X$. 
In particular, the leading term is given by
\begin{align}
(-1)^{|X| + 1}  t^{2|E_X|}  \prod_{v \in V_X \setminus X} ({\rm deg}_{T_X }(v) - 1) . \label{eq_main_det_lead}
\end{align}
\end{thm}
In the case $X = V$, the formula (\ref{eq_main_det_f}) coincides with the binomial expansion of Yan-Yah's formula.

Our formula brings a strong consequence on the signature of $A[X]$. 
Recall that the signature of a symmetric matrix is a pair $(p,q)$ of the number $p$ of positive eigenvalues and the number $q$ of negative eigenvalues. 
When we substitute a sufficiently large value for $t$, the sign of $\det A[X]$ is determined by the leading term. 
By (\ref{eq_main_det_lead}), $\det A[X] > 0$ if $|X|$ is odd, and $\det A[X] < 0$ if $|X|$ is even. 
From Sylvester's law of inertia, the number of sign changes of leading principal minors is equal to the number of negative eigenvalues (see \cite[Theorem 2 in Chapter X]{Gantmacher1959Tto}). 
Therefore the signature of $A[X]$ is $(1,|X| -1)$. 
This argument works on the field $\Rt$ of Puiseux series (defined in Section~\ref{sec_app}).
Thus we have the following.
\begin{cor}\label{cor_signature1}
The signature of $A[X]$ is $(1,|X|-1)$. 
\end{cor} 
In particular, $A[X]$ is nonsingular and defines the Minkowski inner product, i.e., a nondegenerate bilinear form with exactly one positive eigenvalue.

We also consider a skew-symmetric matrix $B = (b_{ij})$ defined by
\[
b_{ij} = -b_{ji} = 
t^{d_{ij}} \quad (i<j).
\]
Denote by $B[X]$ the principal submatrix of $B$ as above. 
In contrast with the symmetric case, the Phaffian 
$\Pf B [X]$ depends on the ordering of $X$. 
We give a simple formula for the case where $X$ has a special ordering, 
though we do not know such a formula for general case. 
A vertex subset $X = \{i_1,i_2,\ldots, i_k\}$ with
$i_1 < i_2 < \cdots < i_k$ 
is said to be {\em nicely-ordered} (with respect to a given tree $T$)
if the tour $i_1 \rightarrow i_2 \rightarrow \cdots \rightarrow i_k \rightarrow i_1$ in $T$
passes through each edge in $T$ at most twice. 
An edge $e$ of $T$ is said to be {\em odd} (with respect to $X$) 
if both two components obtained by
the removal of $e$ from $T$ include 
an odd number of vertices in $X$.
Let $O_{X} \subseteq E$ be the set of odd edges. 
\begin{thm}\label{thm_main_pf}
If $X$ is nicely-ordered and $|X|$ is even, then 
\begin{align}
\Pf B [X]  = t^{|O_{X}|}. \label{eq_thm_main_pf}
\end{align}
\end{thm}

The both formulas are easily generalized to 
a tree metric, that is, a 
dissimilarity matrix that can be embeddable to an edge-weighted tree.
More precisely a {\em dissimilarity matrix} is  a nonnegative symmetric matrix $D \in \QQ^{n \times n}$ with zeros on  diagonal, and 
a {\em tree metric} is a dissimimlarity matrix $D$ 
such that there are a tree $T = (V,E)$, a positive edge weight $l$ on $E$, and a map 
$\varphi: \{1,2,\dots,n \} \to V$ such that 
$D_{ij}$ is equal to the sum of weights of edges of the unique path between $\varphi(i)$ and $\varphi(j)$. 
In this case, 
if $\varphi$ is injective, 
then we can regard $\{1,2,\dots,n \} \subseteq V$, 
and the formula of $\det (t^{D_{ij}})$ 
is obtained by 
replacing $|E_F|$ and $|O_X|$ by weighted sums $\sum_{e \in E_F} l(e)$ and $\sum_{e \in O_X} l(e)$, respectively.
(If $\varphi$ is not injective, then $\det (t^{D_{ij}}) = 0$.) 
The well-known tree metric theorem \cite{Buneman1974Ano} says 
that a dissimilarity matrix $D = (D_{ij})$ is a tree metric if and only if $D$ satisfies
\begin{flalign*}
{\bf [4PC]} \quad \quad \quad \quad D_{ij} + D_{kl} \leq \max \{ D_{ik} + D_{jl}, D_{il} + D_{jk} \} \quad (i,j,k,l \in \{1,2,\dots,n \} ) . &&
\end{flalign*}
This condition is called the {\it four-point condition} [4PC]. 
A symmetric matrix $W= (w_{ij}) \in \QQ^{n \times n}$ satisfying [4PC] (not necessarily a dissimilarity matrix) can be represented with a tree metric $D = (D_{ij})$ and a vector $p = (p_i) \in \QQ^n$ (defined by $p_i := w_{ii} / 2$) as
\begin{align}
w_{ij} = D_{ij} + p_i + p_j .\label{eq_Hij1}
\end{align}
Then we have 
\begin{align}
\det (t^{w_{ij}}) = t^{2 \sum_{k=1}^n p_k} \det(t^{D_{ij}}). \label{eq_Hij2}
\end{align}
Therefore our formulas are applicable to matrices with exponents satisfying [4PC]. 

The organization of this paper is as follows. 
In Section~\ref{sec_app}, we present applications of the formulas.
The space of tree metrics (called the {\em space of phylogenetic trees} in \cite{Billera2001Got}), and related spaces arise ubiquitously from the literature of tropical geometry: 
examples include the tropical grassmannian of rank $2$~\cite{Speyer2004TtG}, the Bergman fan of the matroid of a complete graph~\cite{Ardila2006TBc}, 
and the space of matrices with tropical rank $2$~\cite{Develin2004Tc}. 
In Section~\ref{subsec_trop_one}, we present yet another appearance of the space of tree metrics from tropicalization of the space of Hermite matrices of signature $(1,n-1)$. 
This type of matrix also has interest from theory of multivariate stable polynomials~\cite[Theorem 5.3]{Choe2004Hmp}. 
Recent studies~\cite{Branden2007Pwt, Branden2010DCa, Choe2004Hmp} explored an interesting link between stable polynomials, matroids, and related discrete concave functions.
In Section~\ref{subsection_QPw}, utilizing the formula (\ref{eq_main_det_f}), we establish a correspondence between tree metrics $(D_{ij})$ and quadratic stable polynomials $z^{\top}(t^{D_{ij}}) z$ in $\Rt$.
Our formula also sheds a new insight on 
the {\em dissimilarity map} $X \mapsto |E_X|$ 
of a tree $T =(V,E)$~\cite{Pachter2004Rtf}. 
The dissimilarity map of a tree 
has a significance in phylogenetic analysis as well as has interests from tropical geometry and representation of valuated matroids~\cite{Cools2009Otr, Giraldo2010Dvo, Manon2010Dmo}.
Observe that our leading term formula (\ref{eq_main_det_lead}) gives a new type of representation of the dissimilarity map by the degree of principal minors of a symmetric matrix.
In Section~\ref{subsec_val}, we address this subject. 
In Section~\ref{sec_proofs}, we prove Theorems~\ref{thm_main} and \ref{thm_main_pf}.

\section{Applications}\label{sec_app}
To describe applications of our formulas, let us recall the notion of Puiseux series.
A {\it Puiseux series} in the indeterminate $t$ and a field $K (= {\bf R}, {\bf C})$ is  
a formal series of the form $\sum_{i = i_0}^{ - \infty} a_i t^{i/k}$,
where $i_0$ and $k>0$ are integers and 
each coefficient $a_i$ is an element in $K$. 
Let $K\{t\}$ denote the field of all Puiseux series 
in the indeterminate $t$ and a field $K$.
Define a binary relation $>$ on ${\bf R} \{t\}$ by 
$x > y$ if the leading coefficient of $x - y$ is positive.
By this relation, ${\bf R}\{t\}$ becomes an ordered field. 
Any statement in $\RR$ is naturally formulated in $\Rt$.
From Tarski's principle, any true (first order) statement in $\RR$ is also true in $\Rt$; see Appendix \ref{subsection_first}. 
Hence Corollary~\ref{cor_signature1} is true in $\Rt$.

Let $\bar{\QQ} := \QQ \cup \{- \infty \}$. The {\it valuation} $\val: K \{ t \} \to \bar{\QQ}$ 
is defined by 
\[
\val(x) := \max \{ j/k \mid a_j \neq 0\} \quad \left( x = \sum a_i t^{i/k} \in K\{t\} \right),
\]
where $\val(0) := - \infty$. Namely $\val(x)$ is the degree of the leading term of $x$. 
Define $\val: K \{ t \}^n \to \bar{\QQ}^n$ as $\val (z) := (\val(z_{1}) ,\dots,\val(z_{n})  )$ for $z \in K \{ t \}^n$. 
Through this map, an
algebraic object ${\cal V}$ in $K\{t\}^n$  is transformed to a polyhedral object $\val({\cal V})$  in $\bar {\QQ}^n$, and
an algebraic condition $c_0 z ^{b_0} =  \sum_{i} c_i z^{b_i}$ 
satisfied by ${\cal V}$ 
is transformed to 
a max-plus condition $\val(c_0) + \langle b_0, v\rangle \leq 
\max_{i} \{ \val(c_i) + \langle b_i, v \rangle \}$ satisfied by $\val ({\cal V})$,
which is obtained 
by replacing $(+ , \times )$ with $(\max , +)$ in the original condition.
We will refer to this process 
as a {\em tropicalization}. 
This is a basic idea in tropical geometry;
see~\cite{Speyer2004TtG}.

For $x = y + \sqrt{-1} z \in \Ct$ where $y , z \in \Rt$, the complex conjugate $\overline{x}$ of $x$
is defined as $\overline{x} := y - \sqrt{-1} z$, and $x \overline{x}$ is 
denoted by $|x|^2$.
By an Hermite matrix we mean a matrix $A=(a_{ij})$ with $a_{ij} = \overline{a_{ji}} \in \Ct$.

\subsection{Tropicalizing Hermite matrices with nonnegative
diagonals and signature $(1, n-1)$}\label{subsec_trop_one}
Let ${\cal M}_n$ be the set of $n \times n$ Hermite matrices on $\Ct$ 
having signature $(1,n-1)$ and nonnegative diagonal entries.
Let $\bar{{\cal M}}_n$ be the closure of ${\cal M}_n$, that is, the set of Hermite matrices having nonnegative diagonal entries and at most one positive eigenvalue. 
We regard a symmetric matrix $M = (m_{ij})$ of size $n$ as a vector of dimension $n(n+1) / 2$. 
Then the tropicalization of ${\cal M}_n$ is essentially the space of tree metrics as follows.
\begin{thm}\label{thm_tropS_F}
For a symmetric matrix 
$W = (w_{ij}) \in \bar{\QQ}^{n(n+1/2)}$, the following conditions are equivalent:
\begin{itemize}
\item[{\rm (1)}] $W$ belongs to $\val (\bar{{\cal M}}_n)$.
\item[{\rm (2)}] $W$ satisfies {\rm [4PC]}.
\end{itemize}
In particular, $(t^{w_{ij}}) \in \bar{{\cal M}}_n$ if and only if $W$ satisfies {\rm [4PC]}.
\end{thm}
\begin{proof}
$(1) \Rightarrow (2)$. 
Since $W \in \val (\bar{{\cal M}}_n)$, there is a matrix $M =(m_{ij}) \in \bar{{\cal M}}_n$ such that $w_{ij} = \val (m_{ij})$ for $i,j = 1,2,\dots,n$. 
Every principal submatrix $M[X]$ has at most one positive eigenvalue, and if $M[X]$ has no positive eigenvalue, then $M[X]$ must be a zero matrix 
(since all diagonal entries of $M$ must be zero, 
and all $2 \times 2$ principal minors of $M$ must be nonnegative).
From this we have the following:
\begin{itemize}
\item[($*$)] $\det M[X] \geq 0$ if $|X|$ is odd, and $\det M[X]  \leq 0$ if $|X|$ is even.
\end{itemize}
We show that  $W$ satisfies [4PC]:
\begin{itemize}
\item[(i)] $w_{ii} + w_{kk} \leq 2 w_{ik}$ for distinct $i,k$.
\item[(ii)] $w_{ii} + w_{kl} \leq w_{ik} + w_{il}$ for distinct $i,k,l$.
\item[(iii)] $w_{ij} + w_{kl} \leq \max\{ w_{ik} + w_{jl}, w_{il} + w_{jk} \}$ for distinct $i,j,k,l$.
\end{itemize}

(i): By ($*$) we have $\det M[\{ i,k \}] = m_{ii} m_{kk} - |m_{ik}|^2 \leq 0$.
Since $m_{ii}$ and $m_{kk}$ are nonnegative, 
it must hold that $w_{ii} + w_{kk} = \val(m_{ii} m_{kk} ) \leq \val( |m_{ik}|^2) = 2 w_{ik}$.

(ii): $\det M[\{i,k,l\}] $ is equal to
\begin{align}
 m_{ii} m_{kk} m_{ll} + (m_{ik} \overline{m_{il}} m_{kl} + \overline{m_{ik}} m_{il} \overline{ m_{kl}}) 
- m_{ii}|m_{kl}|^2 - m_{kk}|m_{il}|^2 - m_{ll}|m_{ik}|^2 . \label{detMikl}
\end{align}
From (i), $\val(m_{ii} m_{kk} m_{ll}) \leq \val(m_{ik}\overline{ m_{il}} m_{kl}) = \val( \overline{m_{ik}} m_{il} \overline{m_{kl}})$.
Since $\det M[\{i,k,l\}]   \geq 0$ by ($*$), 
and the last three terms in (\ref{detMikl}) are nonpositive, 
it must hold that 
\[
\val(m_{ii}|m_{kl}|^2) \leq \max\{\val(m_{ii} m_{kk} m_{ll}),  \val (m_{ik} \overline{m_{il}} m_{kl} + \overline{m_{ik}} m_{il} \overline{m_{kl}})\} \leq \val (m_{ik} m_{il} m_{kl}).
\]
Therefore we obtain (ii).

(iii): Consider the expansion of $\det M[\{i,j,k,l\}]$.
For a term containing $m_{i'i'}m_{j'k'}$ in the expansion, 
the term obtained by replacing $m_{i'i'}m_{j'k'}$ with $\overline{m_{i'j'}}m_{i'k'}$ 
also appears in the expansion and has degree at least the original by (i) and (ii). 
From this we see that the degree of a term including a diagonal element $m_{i'i'}$ 
is no more than the degree of $m_{i'j'} m_{j'k'} m_{k'l'} m_{l'i'}$ 
for some different $i',j',k',l'$. 
Observe that $\val (m_{i'j'} m_{j'k'} m_{k'l'} m_{l'i'})$ 
is equal to $(\val (|m_{i' j'}|^2 |m_{k'l'}|^2) + \val(|m_{i' l'}|^2 |m_{j'k'}|^2))/2$.
Therefore, if [4PC] is violated, say, $w_{ij} + w_{kl} > \max\{w_{ik } + w_{jl} , w_{il} + w_{jk}\}$, then $|m_{ij}|^2 |m_{kl}|^2$ becomes the unique leading term in $\det M[\{i,j,k,l\}] $. This implies that $\det M[\{i,j,k,l\}] >0$, which contradicts ($*$). Thus $W$ satisfies [4PC].

$(2) \Rightarrow (1)$. 
It suffices to show that $M := (t^{w_{ij}})  $ belongs to $\bar{{\cal M}}_n$. 
We use the induction on $n$. If $n=1$, then the statement obviously holds. 
Suppose $n>1$. If $M$ is singular, then some $i$-th column (row) can be represented as a linear combination of other column (row). 
Hence the signature of $M$ is equal to
that of the matrix $M'$ obtained by deleting $i$-th column and row; we can apply the induction.
We can assume that $M$ is nonsingular. 
If $w_{ij} = -\infty$ for distinct $i,j$, then [4PC] implies $w_{ik} + w_{jl} \leq w_{il} + w_{jk}$. 
Exchanging the role of $k$ and $l$, we have $w_{ik} + w_{jl} = w_{il} + w_{jk}$. 
This means $m_{ik} m_{jl} = m_{il} m_{jk}$. 
Hence the $i$-th row and the $j$-th row are linearly dependent, and this contradicts the nonsingularity assumption of $M$.
Thus $W$ has $- \infty$ only on diagonals (if it exists). 
By replacing $-\infty$ by a sufficiently small $\sigma \in \QQ$, we obtain $W' = (w'_{ij})$. Then $W'$ satisfies [4PC] (when $\sigma < \min_{i \neq j} \{w_{ij} \}$).
Similarly, $M'$ is defined by $M' := (t^{w'_{ij}})$. 
Then we have $\| M - M' \|_{\infty} \leq t^{\sigma}$. 
From Tarski's principle, the continuity of eigenvalue also holds on $\Rt$. 
Therefore the signatures of $M$ and $M'$ must be the same. 
Hence it is enough to consider the case that $M$ is nonsingular and $W$ has no $-\infty$ entry. 
Then  as the equation (\ref{eq_Hij1}), there are a tree metric $D$ and a vector $p$ such that $w_{ij} = D_{ij} + p_{i} + p_{j}$.
The signatures of $M$ and $(t^{D_{ij}})$ are the same. Since $M$ is nonsingular, the embedding map to the corresponding tree must be injective. Thus we can apply Corollary \ref{cor_signature1} for $(t^{D_{ij}})$, and conclude that the signature of $M$ is $(1,n-1)$.
\end{proof}

\subsection{Quadratic polynomials with the half-plane property}\label{subsection_QPw}
A real multivariate polynomial
$P \in \RR[z_1,z_2,..., z_n]$
is said to {\em have the half-plane property}
if P has no root $z=(z_1, z_2, \dots, z_n)$ with $\Real(z_i) >0 \  (i=1,2,\dots,n)$.
Such a polynomial 
is also called an {\em HPP polynomial}.
Choe, Oxley, Sokal and Wagner~\cite{Choe2004Hmp} and 
Br\"{a}nd\'en~\cite{Branden2007Pwt, Branden2010DCa}
explored an interesting link between the half-plane property 
and matroidal convexity.
A set $B$ of integer vectors in ${\bf Z}^n_{+}$
is called  {\em M-convex} \cite{Murota2003DCA} if $B$ satisfies the following property:
\begin{description}
\item[{[EXC]}]
For $u,v \in B$ and 
$i \in \{1,2,\dots,n \}$ with $u_i < v_i$, 
there exists $j \in \{1,2,\dots,n \}$ such that $u_j > v_j$ and
\[
u + e_i - e_j, v+ e_j - e_i \in B.
\]
\end{description}
An M-convex set is nothing but the base set of an integral polymatroid \cite{Edmonds1970Sfm}. 
In addition, if $B$ belongs to $\{0,1\}^n$, then $B$ is the set of characteristic vectors of bases of a matroid.~\footnote{\cite[Section 7.1]{Choe2004Hmp} refers to an M-convex set as a constant sum jump system.}
An M-convex set $B$ lies on a hyperplane $\{x \in \RR^n \mid \sum_{i=1}^n x_i = k \}$ for some $k \in \ZZ_+$, and this $k$ is called the {\em rank} of $B$. 
The {\em support} of a polynomial $P(z)= \sum a_u z^u$ is the set of vectors $u \in \ZZ_+^n$ such that $a_u \neq 0$, where $z^u := z_{1}^{u_1} \cdots z_{n}^{u_n}$.
\begin{thm}[Choe--Oxley--Sokal--Wagner {\cite[Theorem 7.2]{Choe2004Hmp}}]\label{thm_HPP_Mset}
For every homogeneous HPP polynomial $P$, 
the support of $P$ is an M-convex set. 
\end{thm}
The converse of this theorem is not true in general: there is no HPP polynomial having Fano matroid support~\cite{Branden2007Pwt}.
In rank-$2$ case, however, the following is known.
\begin{thm}[Choe--Oxley--Sokal--Wagner {\cite[Corollary 5.4]{Choe2004Hmp}}]\label{thm_rank2_hpp}
For every M-convex set $B$ of rank $2$, the polynomial $P_B(z)= \sum_{1 \leq i,j \leq n : e_i + e_j \in B} z_i z_j$ has the half-plane property.
\end{thm}
A key ingredient of their proof is the following.
\begin{thm}[Choe--Oxley--Sokal--Wagner {\cite[Theorem 5.3]{Choe2004Hmp}}]\label{prop_half_equiv}
For a nonnegative real symmetric matrix $A$, 
the following conditions are equivalent:
\begin{itemize}
\item[{\rm (1)}] $z^{\top}Az$ has the half-plane property.
\item[{\rm (2)}] $A$ has at most one positive eigenvalue.
\end{itemize}
\end{thm}
Br\"{a}nd\'en \cite{Branden2010DCa} considered HPP polynomials on the field of Puiseux series. 
Since $\Rt$ is an ordered field, the half-plane property is also defined on $\Ct$. 
Namely, $P \in \Rt [z_1, z_2, \dots,z_n]$ is said to have the half-plane property if $P$ has no root $z$ with $\Real(z_i) > 0 \  (i=1,2,\dots,n)$.
For a polynomial $P = \sum a_u z^u$, 
define a function $\val_P$ on $\ZZ_+^n$ by
\[
\textstyle \val_P (u) := \val (a_u) \quad (u \in \ZZ_+^n).
\]
Again $\val_P$ has a matroidal concavity. 
A function $f: \ZZ_+^n \to \bar{\QQ}$ 
is called  {\em M-concave} \cite{Murota2003DCA} if 
\begin{description}
\item[{[M-EXC]}]
for $u,v \in \ZZ_+^n$ and 
$i \in \{1,2,\dots,n \}$ with $u_i < v_i$, 
there exists $j \in \{1,2,\dots,n \}$ such that $u_j > v_j$ and
\[
f(u) + f(v) \leq f(u + e_i - e_j) + f(v_j+ e_j - e_i).
\]
\end{description}
Note that if $f$ is an M-concave function, then the domain of $f$ is the M-convex set~\cite[Proposition 6.1]{Murota2003DCA}, where the {\em domain} is the set of elements $u$ such that $f(u) > -\infty$. Therefore we define the {\em rank} of an M-concave function as the rank of the domain. If the domain of $f$ is contained by $\{0,1\}^n$, then $f$ is a {\it valuated matroid}~\cite{Dress1992Vm}; see Section \ref{subsec_val}. 
\begin{thm}[A corollary of Br\"{a}nd\'en {\cite[Theorem 4]{Branden2010DCa}}]
For every homogeneous HPP polynomial $P$, 
$\val_P$ is an M-concave function.
\end{thm}
We consider the rank-2 case. 
A function $f$ on $\{e_i + e_j \mid 1 \leq i,j \leq n \}$ is identified with a symmetric matrix $(f_{ij})$ by the correspondence
\[
f(e_i + e_j) \longleftrightarrow f_{ij} .
\]
By this correspondence, the condition [M-EXC] for $f$ is equivalent to [4PC] for $(f_{ij})$. 
This fact was observed by Dress and Terhalle~\cite{DressPro1998}, Hirai and Murota~\cite{Hirai2004MFa}. 
Thus Theorem \ref{thm_tropS_F} implies that $A:= (t^{f_{ij}})$ has at most one positive eigenvalue. 
Theorem~\ref{prop_half_equiv} is true in $\Rt$ 
by Tarski's principle (Appendix~\ref{subsection_first}). 
Therefore we have the following.
\begin{cor}\label{them_mainHPP}
For every M-concave function $f$ of rank $2$,
the polynomial 
$P_f (z) =$ $  \sum_{i,j \in [n]} (t^{f(e_i+e_j)}) z_i z_j$ 
has the half-plane property. 
\end{cor}

\begin{rmk}
For a valuated matroid $f$ of rank $2$, the existence of an HPP polynomial $P$ with $\val_P = f$ 
can also be derived from a combination of the following two facts: 
(i) every valuated matroid of rank 2
is representable in $\Rt$~\cite{Speyer2004TtG}, and
(ii) for a reprensentable valuated matroid $f$ 
represented by a $k \times n$ matrix $M$,
the polynomial $Q(z) = \det MZM^{\top}$ is an HPP polynomial 
with $\val_Q = f$, where $Z = \diag (z_1,z_2, \dots, .z_n)$~\cite[Theorem 8.2]{Choe2004Hmp}; 
see Section \ref{subsec_val} for a valuated matroid and its representability.
\end{rmk}

\subsection{Valuated matroid and dissimilarity map on tree}\label{subsec_val}
Our formulas shed some insights on valuated ($\mDelta$-)matroids arising 
from weights of subtrees in a tree.
Denote by ${V\choose k}$ the set of all subsets of $V$ with cardinality $k$. 
For a matrix $M$, denote by $M_X$ the submatrix of $M$ consisting of the $i$-th columns over $i \in X$, and by $M_{X,Z}$ the submatrix consisting of the $i$-th rows and the $j$-th column over $i \in X$ and $j \in Z$.

A {\em valuated matroid} of rank $k$ is 
a map $\omega : {V \choose k} \to \bar{\QQ}$ satisfying
\begin{align*}
\omega (X) + \omega (Y) \leq \max_{j \in Y \setminus X}\{\omega ( X \setminus \{i\} \cup \{j\}) + \omega (Y \setminus \{j\} \cup \{i\})  \} \quad (X, Y \in {\textstyle {V \choose k} },\ i \in X \setminus Y) .
\end{align*}
This condition is the tropicalization of the Grassmann-Pl\"ucker relation of the Pl\"ucker coordinate $v_X := \val(\det M_X)$ for a $k \times n$ matrix $M$:
\[
v_X \cdot v_Y = \sum_{j \in Y \setminus X} \sigma_{ij}  \cdot v_{X \setminus \{i\} \cup \{j\}} \cdot v_{Y \setminus \{j\} \cup \{i\}} \quad (X, Y \in {\textstyle {V \choose k}}, \ i \in X \setminus Y) ,
\]
where $\sigma_{ij} \in \{ 1, -1 \}$ depends on the ordering of the elements $i,j$.
In particular for any $k \times n$ matrix $M$, the map $X \mapsto \val (\det M_X)$ is
a valuated matroid. Such a valuated matroid is called {\em representable}.
In tropical geometry, 
a representable valuated matroid is a point 
of the tropical grassmannian~\cite{Speyer2004TtG}.

In study on phylogenetic trees,
Pachter and Speyer~\cite{Pachter2004Rtf} found that 
weight of subtrees in a tree yields a class of valuated matroids. 
Let $T = (V, E)$ be a tree with a positive edge weight $l$.
For a vertex set $Y$, 
define the {\em dissimilarity} $D(Y)$ of $Y$
by the sum of edge weights $l(e)$ over edges $e$ in
the minimal subtree in $T$ containing $Y$. 
Let $X = \{1,2,\ldots,n\}$ be the set of leaves of $T$.
The {\em k-dissimilarity map} $D^k$ is a function on the $k$-leaf set ${X \choose k}$ defined by $D^k (Y) := D(Y)$.
\begin{thm}[Pachter--Speyer~\cite{Pachter2004Rtf}]
The $k$-dissimilarity map is a valuated matroid.
\end{thm}
Pachter and Speyer~\cite{Pachter2004Rtf} asked whether a $k$-dissimilarity map 
is in the tropical grassmannian, or equivalently, is a representable valuated matroid 
(in our terminology).
Recently this problem was affirmatively solved:
\begin{thm}[Cools~\cite{Cools2009Otr}, Giraldo~\cite{Giraldo2010Dvo}, Manon~\cite{Manon2010Dmo}]\label{thm_kdis_trop}
The $k$-dissimilarity map is a representable valuated matroid.
\end{thm}
Compared with this theorem, our formula (\ref{eq_main_det_f}) gives another type
of a representation of the dissimilarity map $D$ --- 
a representation by the degree of principal minors of a symmetric matrix.
Combinatorial properties
of the map $X \mapsto \val (\det A[X])$ for a symmetric matrix $A$
are not well-studied and not well-understood, though
it is known that
the nonzero support $\{ X \mid \det A[X] \neq 0 \}$ forms a $\mDelta$-matroid~\cite{Bouchet1988RoD, Dress1986Scp}.
A natural question is: does the map $X  \mapsto \val (\det A[X])$
have a kind of a matroidal concavity? 
We hope that our new representation of dissimilarity maps
will stimulate this line of research.

Giraldo~\cite{Giraldo2010Dvo} proved Theorem~\ref{thm_kdis_trop} by showing that 
the total length of a tree is represented as  
the degree of the determinant of a certain matrix associated with the tree.
His formula is somewhat similar to our formula, 
although we could not find a relationship between them.
\paragraph{Representation of rooted $k$-dissimilarity map.}
Nevertheless our formula  gives a linear representation 
for a special class of dissimilarity maps.
Fix a root vertex $0 \in V \setminus X$.
The {\em rooted k-dissimilarity map} $D_0^k$ is 
a function on ${X \choose k}$ defined by
$D_0^k(Y) := D(Y \cup \{0\})$. 
A linear representation of $D_0^k$ is constructed as follows.

Define an $n \times n$ matrix $M = (m_{ij})$
by $m_{ij} := t^{d_{ij}} - t^{d_{0i} + d_{0j}}$.
Namely $M$ is the Schur complement of the $0$-th diagonal element 
in $A[X \cup \{0\}] = (t^{d_{ij}})$.
Hence we have
\begin{itemize}
\item[(1)] $\det M[Y]  = \det A[Y \cup \{0\}]$ for $Y \subseteq X$, and
\item[(2)] $M$ is negative definite.
\end{itemize}
We see (2) from the sign pattern of $\det  M[\{1,2,\ldots,k\}] = \det A[\{0,1,2,\ldots,k\}]$. 
By (2) and the Cholesky factorization, there is an $n \times n$ matrix $Q$ 
with $- M = Q^{\top} Q$.
Take an arbitrary $k \times n$ matrix $J$ in $\RR$ 
whose entries have no algebraic dependence.
By the Binet--Cauchy formula we have
\begin{eqnarray*}
&& 2 \val (\det (JQ)_{Y})  =  2 \val( \sum_{Z} \det J_Z \det Q_{Z,Y})  
 =  2 \max_{Z}\val ( \det Q_{Z,Y}  ) \\
 &&  = \max_{Z} \val(  (\det Q_{Z,Y})^2 )=  \val( \sum_{Z}  (\det Q_{Z,Y})^2 ) =   \val( \det (Q_{Y})^{\top} Q_{Y}) \\ 
&& =  \val ( \det  M[Y])  =  \val( \det A[Y \cup \{0\}] )= D^{k}_0(Y),
\end{eqnarray*}
where $Z$ ranges all elements in ${X \choose k}$, and the second equality follows from the algebraic independence of elements in $J$.
Hence let $R := JQ$ and replace $t$ by $t^{1/2}$ in $R$. 
Then $D^{k}_0(Y) = \val( \det R_Y)$,
and we obtain a linear representation of $D^k_0$.

\paragraph{Valuated $\mDelta$-matroid and odd-dissimilarity map.}
A {\em valuated $\mDelta$-matroid}~\cite{Dress1991Agc, Wenzel1993Pfa} is 
a function $\omega : 2^V \to \bar{\QQ} $ satisfying
\[
\omega (X) + \omega (Y) \leq \max_{j \in (X \triangle Y) \setminus i}\{ \omega(X\triangle \{i,j\}) + \omega(Y \triangle \{i,j\}) \} \quad (X,Y \subseteq V,\ i \in X \triangle Y)  .
\]
This is the tropicalization of the Wick relation of principal minors $b_X := \Pf B[X] \ (X \subseteq V)$ of a skew-symmetric matrix $B$: 
\[
b_X \cdot b_Y = \sum_{j \in (X \triangle Y) \setminus i} \sigma '_{ij} \cdot b_{X\triangle \{i,j\} } \cdot b_{ Y \triangle \{i,j \} } \quad (X,Y \subseteq V,\ i \in X \triangle Y) ,
\]
where $\sigma '_{ij} \in \{ 1, -1 \}$ depends on the ordering of the elements $i,j$. 
Hence the map $X \mapsto \val(b_X)$ is a valuated $\mDelta$-matroid~\cite{Wenzel1993Pfa}; see also~\cite[Section 7.3]{Murota2000MaM}. Such a valuated $\mDelta$-matroid is called {\em representable}.

Let $T=(V,E)$ be a tree where $V = \{ 1,2,\dots,N \}$. 
For any tree $T$, the {\em odd-dissimilarity map} $D^o \in 2^V$ is defined as follows.
\[
D^o(X) := \left \{
\begin{array}{ll}
|O_X| &\text{if $|X|$ is even,}\\
-\infty &\text{if $|X|$ is odd,}
\end{array}
\right. \quad (X \subseteq V),
\] 
where $O_X$ is the set of odd edges with respect to $X$, defined in Section \ref{sec_intro}. 
After reordering, we suppose that $V$ is nicely-ordered. One can easily see that any subset $X \subseteq V$ is also nicely-ordered. By (\ref{eq_thm_main_pf}), we have 
\[
D^o (X) = \val ( \Pf B[X]) \quad (X \subseteq V).
\]
Moreover, let $B^{\vee}$ be the matrix obtained by replacing $t$ by $t^{-1}$ in $B$. Then we have
\[
- D^o (X) = \val ( \Pf B^{\vee}[X]) \quad (X \subseteq V).
\]
Therefore we obtain the following.
\begin{cor}\label{cor_odd_dissimi}
The odd-dissimilarity map and its negative are both representable valuated $\mDelta$-matroids.
\end{cor}
This theorem implies that the odd-dissimilarity map is a nontrivial example of a valuated $\mDelta$-matroid whose negative is also a valuated $\mDelta$-matroid.

An algebraic variety determined by the Wick relation is called the {\em spinor variety}.
The spinor variety parametrizes maximal isotropic vector subspaces
in a vector space with an antisymmetric bilinear form, analogous to the grassmannian that parametrizes vector subspaces.
Rinc{\'o}n~\cite{Rincon2012Ils} considered the {\em tropical spinor variety} 
(a tropicalization of the spinor variety).  A representable valuated $\mDelta$-matroid is nothing but a point of  the tropical spinor variety in his sense.
Corollary~\ref{cor_odd_dissimi}
is therefore an isotropic analogue of Theorem~\ref{thm_kdis_trop}.

\section{Proof}\label{sec_proofs}
\subsection{Proof of Theorem \ref{thm_main}}
Let $T=(V,E)$ be a tree, and $X \subseteq V$. Let us recall the formula for the determinant of $A[X]$. Without loss of generality, we can assume that $X=\{1,2,\dots, n\}$.
\[
\det A[X] = \sum_{\sigma \in S_{n}} \sign(\sigma) \prod_{i=1}^{n}a_{i\sigma(i)},
\]
where $S_{n}$ is the symmetric group of degree $n$. Our first step is to identify each permutation of this formula with a path on the corresponding tree. 
Let us define following terminology.
\begin{itemize}
\item A {\em cycle} of $X$ is a cyclic sequence 
$(x_1, x_2, \ldots, x_k)$ $(k \geq 1)$ 
of a subset $\{x_1, x_2, \ldots, x_k\} \subseteq X$, 
where we do not distinguish $(x_1, x_2, \ldots, x_k)$  and  
$(x_k, x_1,x_2,\ldots, x_{k-1})$, and indices are considered modulo $k$.
\item A {\em cycle partition} $W$ of $X$  is a set of cycles of $X$ 
such that every element of $X$ belongs to exactly one cycle.
\item The {\it support} ${\rm supp}(C)$ of  a cycle $C=(x_1,x_2, \ldots, x_k)$ is a function on $E$ defined by: ${\rm supp}(C)(e)$ is 
the number of indices $i$ such that
$e$ belongs to the unique path between $x_i$ and $x_{i+1}$ in $T$.  
The support ${\rm supp}(W)$ of a cycle partition $W$ is 
defined as $\sum_{C \in W} {\rm supp}(C)$. Note that the support is even-valued. 
\item $\sign (W) := \prod_{C \in W} (-1)^{|C| + 1}$.
\item $\|  W \| := \sum_{e \in E} {\rm supp} (W)(e)$. 
\end{itemize}
For a cycle $C = (i_1,i_2,\ldots,i_k)$, this definition means that 
\begin{align}
\sum_{j =1}^k d_{i_j  i_{j+1}} = \sum_{e \in E} {\rm supp} (C) (e) . \label{eq_|C|}
\end{align}
By using these notions, the formula of the determinant can be rephrased as follows.
\begin{lem}\label{lem_32}
\begin{align}
\det  A[X] = \sum \left \{ \sign (W) t^{\|W\|} \mid \text{$W$: cycle partition of $X$}\right \}. \label{eq_lem_32}
\end{align}
\begin{proof}
Observe that there is a one-to-one correspondence between permutations 
and cycle partitions: 
a permutation is uniquely represented as the product of (disjoint) cyclic permutations, 
and each cyclic permutation $(i_1,i_2,\ldots, i_k)$ is 
naturally identified with a cycle in our sense.
In this correspondence, the sign of a permutation $\sigma$ 
is equal to $\sign (W)$ of the corresponding $W$, and by the equation~(\ref{eq_|C|}), we have
\begin{align*}
\prod_{i \in X}a_{i\sigma(i)} &=  t^{\sum_{i \in X} d_{i\sigma(i)}} = t^{\|  W \|}.
\end{align*}
Hence we have
\begin{align*}
\det A[X] & = \sum_{\sigma \in S_{n}} \sign(\sigma) \prod_{i \in X}a_{i\sigma(i)}
= \sum_{W} \sign (W)  t^{\|  W \|}.
\end{align*}
\end{proof}
\end{lem}
A cycle partition $W$ of $X$ 
is said to be {\em tight} if the support of $W$ 
is $\{0,2\}$-valued. 
In fact, non-tight cycle partitions vanish in (\ref{eq_lem_32}).
\begin{lem}\label{lem_main_lem}
\begin{align*}
\det A[X] 
= \sum\left \{ \sign (W) t^{\|W\|} \mid \text{$W$: tight cycle partition of $X$}\right \} . 
\end{align*}
\end{lem}
\begin{proof}Let us first introduce an operation on cycle partitions, called a {\em flip}.
Let $W$ be a cycle partition of $X$, and let $e = xy$ be an edge of $T$.
Suppose that ${\rm supp}(W)(xy) \geq 4$. 
Then (i) there are  two cycles $C,C'$ passing through $e$ 
in order $x \to y$, or (ii) there is a single cycle $C''$ 
passing through in order $x \to y$ twice.
For the case (i), suppose that $C = (v_1,v_2,\dots, v_k)$, $C' = (u_1,u_2,\dots,u_l)$, 
the unique path from $v_i $ to  $v_{i+1}$ passes through $xy$ in order $v_i \to x \to y \to v_{i+1}$, and
the unique path from $u_j $ to  $u_{j+1}$ passes through $xy$ in order $u_j \to x \to y \to u_{j+1}$.
Replace $C$ and $C'$ by  
\begin{equation}
C'' = (v_1,\ldots, v_i, u_{j+1},\dots,u_l,u_1 \ldots,u_j, v_{i+1},\ldots, v_k).
\end{equation}
Then we obtain a cycle partition $W' = W \setminus \{C,C'\} \cup \{C''\}$. 
Similarly, for the case (ii), there is a single cycle $C''$ as in (\theequation).
By reversing the operation above, we obtain two cycles $C,C'$.
Replacing $C''$ by $C$ and $C'$, 
we obtain a cycle partition $W' = W \setminus \{C''\} \cup \{C, C'\}$.
In this way, we obtain an operation $W \mapsto W'$ on cycle partitions, 
which we call {\em a flip}.

If a cycle partition $W'$ is obtained by applying a flip to another cycle partition $W$, then $\supp (W') = \supp (W)$ and 
$\sign(W') = - \sign(W)$. 
The former equation is obvious from the definition of a flip. The latter equation holds since $|W| - |W'| \in \{ 1, -1  \}$ and 
\begin{align*}
\sign(W) &= \prod_{C \in W} (-1)^{|C| + 1} = (-1)^{|X| + |W|}.
\end{align*}

For $l:E\to 2{\bf Z}_+$, let ${\cal W}_l$ be the set of all cycle partitions $W$ with ${\rm supp }(W) = l$. Let ${\cal W}_l^+$ denote 
the set of cycle partitions $W$ 
in ${\cal W}_l$ with $\sign (W) = 1$, 
and let ${\cal W}_l^- := {\cal W}_l \setminus {\cal W}_l^+$.
Then, from (\ref{eq_lem_32}), we have
\begin{align}
\det A[X] = \sum_{l:E \to 2{\bf Z}_+} (  |{\cal W}^+_l| - |{\cal W}^-_l | )  t^{\|l\|}, \label{eq_lem_+-}
\end{align}
where $\|l\| := \sum_{e \in E} l(e)$. 
It suffices to show that if there is an edge $e \in E$ with $l(e) \geq 4$, then $|{\cal W}_l^{+}| = |{\cal W}_l^{-}|$.

Let $\Gamma$ be the graph on ${\cal W}_l$ such that two vertices $W, W' \in {\cal W}_l$ are adjacent if and only if $W$ can be obtained from $W'$ by a single flip on $e$. 
The graph $\Gamma$ is a bipartite graph of bipartition $\{ {\cal W}_l^{+} , {\cal W}_l^{-} \}$, since a flip operation changes the sign, and 
$\Gamma$ cannot have an edge joining vertices of the same sign. Moreover  $\Gamma$ is a regular graph, 
since the number of different flips 
on $e$ is determined only by $l(e)$ 
(, which is equal to $2 \binom{l(e)/2}{2}$), 
and different flips yield different cycle partitions. 
Thus $\Gamma$ is a regular bipartite graph with bipartition $\{ {\cal W}_l^{+} , {\cal W}_l^{-} \}$, which implies $|{\cal W}_l^{+}| = |{\cal W}_l^{-}|$.
\end{proof}

For any set ${\cal W}$ of cycle partitions, define the number $\langle{\cal W}\rangle$ by 
\[
\langle{\cal W}\rangle:= \sum_{W \in {\cal W}} {\rm sign} (W).
\]
For a forest $F$ (not necessarily a subgraph of $T$) spanned by $X$, 
define ${\cal W}_{X,F}$ by the set 
of cycle partitions $W$ on $X$
such that each cycle $C \in W$ 
belongs to some connected component of $F$, 
and each edge in $F$ is traced by cycles in $W$ exactly twice. 
By using this notation, we have the following.
 \begin{lem}\label{lem_main_lem2}
\begin{align}
\det A[X] = \sum_{F} \langle {\cal W}_{X,F}\rangle  
 t^{2|E_F|}, \label{eq_lem_main_lem2}
\end{align}
where $F$ ranges all subgraphs in $T$ spanned by $X$.
\end{lem}
\begin{proof}
This immediately follows from Lemma \ref{lem_main_lem} and the fact that
for any tight cycle partition $W$
the forest formed by edges $e$ with ${\rm supp} (W)(e) = 2$ is spanned by $X$.
\end{proof}

Hence, to prove Theorem~\ref{thm_main}, it suffices to show the following:
\begin{lem}\label{prop_DFX}
\[
\langle {\cal W}_{F,X}\rangle  = ( - 1)^{|X| + c(F) }\prod_{v \in V_{F} \setminus X} ({\rm deg}_F (v) - 1).
\]
\end{lem}
This lemma is an easy corollary of the following properties of $\langle {\cal W}_{X,F}\rangle$.
\begin{lem}\label{lem_DFX} \begin{itemize}
\item[{\rm (i)}]Suppose that $F$ is the vertex-disjoint union of two forests $H,H'$. Then we have 
\[
\langle{\cal W}_{F,X}\rangle = \langle{\cal W}_{H,X \cap V_{H}}\rangle \langle{\cal W}_{H',X \cap V_{H'}}\rangle.
\]
\item[{\rm (ii)}]For $e = xy \in E_{F}$,
let $F'$ be the forest obtained from $F$ by adding new vertices $x',y'$ and by
replacing $e$ by new edges $xy', x'y$. Then $F'$ is spanned by $X \cup \{x',y'\}$, and  we have
 \[
\langle{\cal W}_{F,X}\rangle  =  - \langle{\cal W}_{F', X\cup \{x',y'\}}\rangle.
\]
\item[{\rm (iii)}]If $F$ is a star with the center vertex $v$, then
\[ 
\langle{\cal W}_{F,X}\rangle   =  \left\{ 
\begin{array}{ll}
(-1)^{|X| + 1} & {\rm if} \ \mbox{$v \in X$}, \\
(-1)^{|X| + 1} ({\rm deg}_F (v) - 1) & {\rm otherwise}.
\end{array} \right.
\]
\end{itemize}
\end{lem}
\begin{proof}[Proof of Lemma \ref{prop_DFX}]
By Lemma \ref{lem_DFX} (i), it suffices to prove the formula for the case where $F$ is connected. 
We use the induction on the number $k$ of non-leaf vertices. 
If $k=0$ or $1$, then $F$ is a star, and the corresponding formula follows from (iii). 
Let $k>1$. Since $F$ is connected, there exists an edge $e$ joining two non-leaf vertices. Applying (ii) for $e$, we have $\langle{\cal W}_{F,X}\rangle =  - \langle{\cal W}_{F', X \cup \{x',y'\}}\rangle$, where $F'$ has two connected components $H$, $H'$, each of which has less non-leaf vertices than $F$ has. Let $Y:= (X \cup \{x',y'\}) \cap V_{H}$, and $Y' := (X \cup \{x',y'\}) \cap V_{H'}$. From (i) and inductive hypothesis, we get
 \begin{align*}
\langle{\cal W}_{F,X}\rangle &=  - \langle{\cal W}_{F', X \cup \{x',y'\}}\rangle =-\langle{\cal W}_{H, Y}\rangle \langle{\cal W}_{H', Y'}\rangle \\
&= - (-1)^{| X \cup \{x',y'\}| + |V_{H}| + |V_{H'}|+ |E_{H}| + |E_{H'}| } \prod_{v \in V_{F'} \setminus (X \cup \{x',y'\})} ({\rm deg}_F (v) -1)\\
&=(-1)^{|X|  + |V_F|- |E_F|} \prod_{v \in V_{F} \setminus X} ({\rm deg}_F (v) -1),
\end{align*}
where $|V_{H}| + |V_{H'}| = |V_F| + 2 $ and $|E_{H}| + |E_{H'}| = |E_F| +2 $. Since $c(F) = |V_F|- |E_F|$, we have the desired equation.
\end{proof}
\begin{proof}[Proof of Lemma \ref{lem_DFX}]
(i) Since every cycle partition $W \in {\cal W}_{F,X}$ is uniquely decomposed into cycle partitions $Z \in {\cal W}_{H,X \cap V_H}$ and $Z' \in {\cal W}_{H',X \cap V_{H'}}$ with $W = Z \cup Z'$, and vice versa, we have
\begin{align*}
\langle{\cal W}_{F,X}\rangle &= \sum_{W \in {\cal W}_{F,X}} {\rm sign} (W) = \sum_{Z\in {\cal W}_{H,X \cap V_{H}}} \  \sum_{Z'\in {\cal W}_{H',X \cap V_{H'}}}(\sign(Z))(\sign(Z'))\\
&= \langle{\cal W}_{H,X \cap V_H}\rangle \langle{\cal W}_{H',X \cap V_{H'}}\rangle.
\end{align*}

(ii) For every cycle partition $W \in {\cal W}_{F,X}$, 
there is the unique cycle 
\[
C = (u,v, \alpha_1, \dots ,\alpha_i, v',u', \beta_1,\dots,\beta_j) \in W
\]
such that the path between $u,v$ and the path between $v',u'$
include $xy$ in order $u \rightarrow x, y \rightarrow v$ and $v' \rightarrow y, x \rightarrow u'$, respectively. 
Define two cycles $C', C''$ by 
\begin{equation}
C' := (u,y',u', \beta_1,\dots,\beta_j), \quad C'' := ( v', x',v,\alpha_1,\dots,\alpha_i).
\end{equation}
Let $W' := W \setminus \{C\} \cup \{C',C' \}$.
Then $W' $ is a cycle partition in ${\cal W}_{F', X\cup \{x',y'\}}$ with $
\sign (W) = - \sign (W')$. 
Thus we obtain a map from ${\cal W}_{F,X}$ to ${\cal W}_{F', X\cup \{x',y'\}}$ such that $W \mapsto W'$.
Observe that this map is a bijection; any cycle partition $W' \in {\cal W}_{F', X\cup \{x',y'\}}$ 
includes cycles $C,C'$ with property (\theequation), 
and the reverse operation is definable on every cycle partition.
Hence we obtain
\begin{align*}
 \langle{\cal W}_{F , X \cap V_{F}}\rangle = \sum_{W \in {\cal W}_{F,X}} \sign(W) = -\sum_{W'\in {\cal W}_{F', X\cup \{x',y'\}}}\sign(W') =  - \langle{\cal W}_{F', X \cup \{x',y'\}}\rangle.
\end{align*}

(iii) Let $k := |X|$. In the both cases, $\langle{\cal W}_{F , X }\rangle$ depends only on the cardinality of $X$. We may denote ${\cal W}_{F , X }$ by ${\cal A}_k$ if $v \not\in X$, and denote ${\cal W}_{F , X }$ by ${\cal B}_k$ if $v \in X$. We will prove the following two claims.
\begin{itemize}
\item[($*1$)] $\langle{\cal A}_k \rangle = -(k-1)(\langle{\cal A}_{k-1} \rangle  + \langle{\cal A}_{k-2} \rangle ), \quad (k>3) .$
\item[($*2$)] $\langle{\cal B}_k \rangle  = \langle{\cal A}_{k} \rangle  + \langle{\cal A}_{k-1} \rangle , \quad (k>2) .$
\end{itemize}
By $\langle{\cal A}_2 \rangle = -1$, $\langle{\cal A}_3 \rangle = 2$, and the recursion ($*1$), we have
\begin{align*}
\langle{\cal A}_k \rangle 
= (-1)^{k+1}(k-1)
= (-1)^{|X| + 1} ({\rm deg}_F (v) - 1).
\end{align*}
Also we have
$\langle{\cal B}_1 \rangle= 1$, 
$\langle{\cal B}_2 \rangle = -1$, and from ($*2$),
$\langle{\cal B}_k \rangle = \langle{\cal A}_{k} \rangle  + \langle{\cal A}_{k-1} \rangle  = (-1)^{|X|+1}$.

For ($*1$), fix an arbitrary vertex $u \in X$. 
Let ${\cal A}'_k$ denote the set of 
cycle partitions $W$ in ${\cal A}_k$ such that
the unique cycle in $W$ containing $u$ has length at least three.
We will show that 
\begin{align}
\langle {\cal A}_k' \rangle &= - (k-1)\langle{\cal A}_{k-1} \rangle ,\label{equation_a_k_1}\\
(\langle {\cal A}_k \rangle - \langle {\cal A}_k' \rangle =) \langle {\cal A}_k \setminus {\cal A}_k' \rangle &= - (k-1)\langle{\cal A}_{k-2} \rangle .\label{equation_a_k_2}
\end{align}

To see (\ref{equation_a_k_1}), for a cycle partition $W \in {\cal A}_{k-1}$, 
take a consecutive pair $x,y$ in some cycle $C$ in $W$. 
Replace $x,y$ by $x,u,y$ in $C$. 
Then we get a cycle partition $W'$ 
in ${\cal A}_k'$  with sign change.
There are $k-1$ ways of choosing a consecutive pair in each cycle partition. 
Also every cycle partition in ${\cal A}_k'$ is obtained in this way. Hence we have (\ref{equation_a_k_1}).

To see (\ref{equation_a_k_2}), observe that ${\cal A}_k \setminus {\cal A}_k'$ is the disjoint union, over $x \in X \setminus u$, of  
the sets ${\cal W}_{x}$ of cycle partitions including $(u,x)$. 
Delete $(u,x)$ from each cycle partition of ${\cal W}_{x}$. Then we get a cycle partition  
in ${\cal A}_{k-2}$ with sign change. Also, every cycle partition in ${\cal W}_{x}$ is obtained by adding the cycle $(u,x)$ to cycle partitions in ${\cal A}_{k-2}$. Hence 
$\langle {\cal A}_k \setminus {\cal A}_k' \rangle = \sum_{x \in X \setminus u} \langle{\cal W}_{x}\rangle = -(k-1)\langle {\cal A}_{k-2} \rangle$, 
and we have (\ref{equation_a_k_2}).

Consider ($*2$). Let ${\cal B}_k '$ denote the set of cycle partitions $W$ in ${\cal B}_k$ such that $W$ includes the singleton cycle $(v)$. 
It suffices to show that $\langle {\cal B}_k ' \rangle  = \langle {\cal A}_{k-1} \rangle$ and $\langle {\cal B}_k  \setminus {\cal B}_k ' \rangle = \langle {\cal A}_k \rangle$.

The first equation follows from the observation that the deletion of $(u)$ from cycle partitions in ${\cal B}_k ' $ makes a one-to-one correspondence between ${\cal B}_k ' $  and ${\cal A}_{k-1}$. For the latter equation, add a new leaf $v'$ to $F$, and replace $v$ by $v'$ in each cycle partition in ${\cal B}_k  \setminus {\cal B}_k '$. This procedure maps cycle partitions in ${\cal B}_k  \setminus {\cal B}_k '$ to ones in ${\cal A}_k$ bijectively, and thus we have the latter equation.
\end{proof}

\subsection{Proof of Theorem \ref{thm_main_pf}}
Suppose that $X=\{1,2,\dots, 2n\}$ and $X$ is nicely-ordered with respect to $T$. We denote $\Pf B[X]$ by $\Pf[X]$ for simplicity. Let us recall the recursive definition of Pfaffian:
\begin{align}
\Pf [X] = \sum_{j \in X} (-1)^{i+j+1} b_{ij} \Pf [X \setminus \{ i ,j \}]  \quad (i \in X).\label{eq_lem_pf_x}
\end{align}
Since the deletion of an element in $X$ only omits paths of the corresponding tour, we have the following lemma.
\begin{lem}
If $X$ is nicely-ordered, then every subset $Y$ of $X$ is nicely-ordered.
\end{lem}
In the following, we tacitly use this lemma.
For distinct $i,j \in X$, define $P_{ij} \subseteq E$ as the set of edges which belong to the unique path from $i$ to $j$.
\begin{lem} \label{lem_sym_diff}
$ O_{X\setminus \{i,j\}} = O_{X} \triangle P_{ij}$.
\begin{proof}
Let $e \in E$, and let $T'$, $T''$ be the two components obtained by the removal of $e$. 
If $e \not\in P_{ij}$, then either $T'$ or $T''$ must include both $i$ and $j$, and hence $e \in O_{X\setminus \{i,j\}} \Leftrightarrow e \in O_{X}$. 
If $e \in P_{ij}$, then $T'$ must include just one of $i$ and $j$, and hence $e \in O_{X\setminus \{i,j\}} \Leftrightarrow e \not \in O_{X}$. These imply the statement.
\end{proof}
\end{lem}

The following lemma gives a pairing of elements of $X$ via odd edges.
\begin{lem}\label{prop_pf_cancel}
There is a partition $\{ \{i_1 , j_1\}, \{i_2, j_2\}, \dots, \{i_n, j_n\} \}$ of $X$ such that 
\begin{itemize}
\item[{\rm (i)}] $i_k + j_k $ is odd for all $k=1,\dots,n$, and
\item[{\rm (ii)}] $O_{X}$ is the disjoint union of $P_{i_1 j_1} , \dots, P_{i_n j_n}$. In particular, it follows that 
\[
P_{u i_k} \setminus O_{ X } = P_{u j_k} \setminus O_{X } \quad (u \in X ,\ k = 1,2,\dots,n ).
\]
\end{itemize}
\begin{proof}
We first show that there exists $i$ with $P_{i(i+1)} \subseteq O_{X}$, 
where the indices are considered modulo $2n$.
We may assume that all leaves of $T$ belong to $X$. 
Consider the subgraph $H$ of $T$ formed by $O_X$. 
There exists a connected component $T'$ of $H$ 
incident to (at most) one edge $e \in E \setminus O_X$ in $T$. 
Necessarily $T'$ contains at least two vertices $i,j$ in $X$. 
Then $i-1$ or $i+1$ also belongs to $T'$; 
otherwise the edge $e$ is traced at least four times by the tour 
$1 \rightarrow 2 \rightarrow \cdots \rightarrow 2n \rightarrow 1$; 
contradiction to the fact that $X$ is nicely-ordered. 
This implies $P_{(i-1)i}  \subseteq O_{X}$ or $P_{i(i+1)}  \subseteq O_{X}$.

We prove the statement of this lemma by induction on the cardinality
of $X$.
Pick a vertex $i$ with $P_{i(i+1)} \subseteq O_{X}$.
Let $\{i_1 , j_1\} := \{i,i+1\}$.
Then $O_X$ is the disjoint union of $P_{i(i+1)}$ and 
$O_{X \setminus \{i,i+1\}}$, and $i_1+j_1$ is odd. 
We can renumber $X \setminus \{ i,i+1 \}$ with keeping the parity of the indices.
By induction, $X \setminus \{ i,i+1 \}$ has a partition
$\{ \{i_2,j_2\},\dots,\{ i_n, j_n\}  \}$
such that $i_k + j_k$ is odd for all $k=2,\dots,n$ and
$O_{X \setminus \{i,i+1\}}$ is the disjoint union of $P_{i_2 j_2} , \dots, P_{i_n j_n}$.
Then $\{ \{i_1 , j_1\}, \{i_2, j_2\}, \dots, \{i_n, j_n\} \}$ is a desired partition of $X$, and the proof is complete.
\end{proof}
\end{lem}
We are ready to prove Theorem \ref{thm_main_pf}.
\begin{proof}[Proof of Theorem \ref{thm_main_pf}]
We prove the statement by the induction on the cardinality of $X$. If $X = \{1, 2\}$, 
then $\Pf[\{1, 2\}] = b_{12} = t^{d_{12}} =t^{|O_{\{1,2\}}|}$. 
Suppose $|X| > 2$. Fix a partition $\{\{i_1, j_1\}, \dots, \{i_n, j_n\}\}$ of $X$ satisfying the condition in Lemma \ref{prop_pf_cancel}. We can assume that $i_1 = 1$.
Since $j_1$ is even and $i_k + j_k $ is odd for all $k$, from the formula (\ref{eq_lem_pf_x}) we have
\begin{align}
\Pf[X] =& b_{1 j_1}\Pf [X  \setminus \{ 1, j_1\} ] \nonumber \\
&+ \sum_{k=2}^n 
 (-1)^{i_k} \Bigl( b_{1 i_k} \Pf [X  \setminus \{ 1, i_k\} ]  -  b_{1 j_k} \Pf[X  \setminus \{ 1, j_k\} ] \Bigr) . \label{eq_th_pf_0}
\end{align}
Since $O_{X }$ is the disjoint union of $P_{1 j_1}$ and $O_{ X  \setminus \{ 1, j_1\}}$, by inductive hypothesis we have
\begin{align*}
b_{1 j_1} \Pf [X  \setminus \{ 1, j_1\} ] &= t^{|P_{1j_1}| + |O_{ X  \setminus \{ 1, j_1\}}|} = t^{|O_{ X }|} ,\\
b_{1 i_k} \Pf [X  \setminus \{ 1, i_k\} ] &=t^{|P_{1i_k}| + |O_{ X  \setminus \{1, i_k\}}|},\\
b_{1 j_k} \Pf[X  \setminus \{ 1, j_k\} ] &=  t^{|P_{1j_k}| + |O_{X  \setminus \{ 1,j_k\}}|} .
\end{align*}
From Lemma \ref{lem_sym_diff}, we have
$
|P_{1i_k}| + |O_{ X  \setminus \{1, i_k\}}|  = |P_{1 i_k}| + |O_{ X } \triangle P_{ 1i_k}|
= |O_{ X } | + 2|P_{1i_k} \setminus O_{ X }|
= |O_{ X } | + 2|P_{ 1j_k} \setminus O_{ X }|
= |P_{1j_k}| + |O_{X  \setminus \{ 1,j_k\}}|
$. 
Hence the sum of the equation (\ref{eq_th_pf_0}) vanishes, and we have (\ref{eq_thm_main_pf}).
\end{proof}

\subsection*{Acknowledgements}
The authors thank Kazuo Murota for drawing the author's attention to this subject,
and also thank the referees for helpful comments.
This research was partially supported by KAKENHI (21360045, 23740068, 26330023).

\bibliographystyle{fplain}
\bibliography{all}

\appendix
\section{Tarski's principle for real closed fields}\label{subsection_first}
A field $K$ is a {\em real closed field} if $K$ is an ordered field such that every positive element is a sum of squares in $K$, and every polynomial on $K$ of odd degree has at least one root in $K$ (see~\cite[p. 34]{Basu2006Air}). 
It is known that $\Rt$ is a real closed field (see~\cite[Theorem 2.91]{Basu2006Air}). 
An important fact in a real closed field is the following:
\begin{thm}[Tarski's principle (see~{\cite[Theorems 2.80, 2.81]{Basu2006Air}})]
A first-order statement is 
true in a real closed field if and only if it is true in every real closed field.
\end{thm}
Here a first-order statement is a predicate constructed from addition, multiplication, equality,
inequality, and the standard logical connectives and quantifiers. 
Hence any true first order statement in $\RR$ is also true in $\Rt$. 
For example, the statement ``a polynomial $P(z)$ in ${\bf R}\{t\}$ is HPP" 
 can be written as a first-order statement in $\Rt$ as follows.
Substitute $u + i v$ to $z$ in $P(z)$, and
represent $P$ 
as $P (u + iv) = Q(u,v) + i R(u,v)$, 
where $Q, R$ are  polynomials in ${\bf R}\{t\}$.
Then the HPP statement is equivalent to
\[
\forall u \forall v ( Q(u,v) = 0 \wedge R(u,v) = 0 \rightarrow \lnot(u \geq 0)).
\] 
In this way, 
any polynomial relation in $\Ct$ can be written in polynomial relations in $\Rt$. 
Therefore the statement ``an Hermite matrix has real eigenvalues only" and Sylvester's law hold
in $\Ct$, which were used in Section~\ref{subsec_trop_one}. 
Also Theorem~\ref{prop_half_equiv} in Section~\ref{subsection_QPw} holds in $\Rt$.
\end{document}